\documentclass[12pt,reqno]{amsart}
\usepackage{amsmath, amsthm, amssymb, stmaryrd}

\advance \topmargin by -\headheight
\advance \topmargin by -\headsep
     
\evensidemargin \oddsidemargin
\newtheorem{theorem}{Theorem}[section]
\newtheorem{lemma}[theorem]{Lemma}
\newtheorem{corollary}[theorem]{Corollary}

\theoremstyle{definition}

\theoremstyle{remark}

\numberwithin{equation}{section}

\newcommand{\mmod}[1]{\,\,(\text{\rm mod}\,\,#1)}

\def\bfx{{\mathbf x}}

\def\calA{{\mathcal A}}  

\def\calB{{\mathcal B}}

\def\calD{{\mathcal D}}

\def\calN{{\mathcal N}}

\def\calU{{\mathcal U}}

\def\Ftil{\widetilde{F}}

\def\ftil{\widetilde{f}}

\def\Util{\widetilde U}

\def\dbZ{{\mathbb Z}}

\def\grB{{\mathfrak B}}

\def\grm{{\mathfrak m}}\def\grM{{\mathfrak M}}
\def\grN{{\mathfrak N}}\def\grn{{\mathfrak n}}

\def\grP{{\mathfrak P}}
\def\grB{{\mathfrak B}}

\def\alp{{\alpha}} 
\def\bet{{\beta}}

\def\del{{\delta}} \def\Del{{\Delta}}

\def\tet{{\theta}} \def\bftet{{\boldsymbol \theta}} 
 
\def\kap{{\kappa}}

 \def\Sig{{\Sigma}} 

\def\Ups{{\Upsilon}}

\def\ome{{\omega}} \def\Ome{{\Omega}}

\def\eps{\varepsilon}

\def\le{\leqslant} \def\ge{\geqslant}

\def\d{{\,{\rm d}}}

\begin{document}
\title[Smooth Weyl sums over biquadrates]{On smooth Weyl sums over biquadrates\\ and 
Waring's problem}
\author[J\"org Br\"udern]{J\"org Br\"udern}
\address{Mathematisches Institut, Bunsenstrasse 3--5, D-37073 G\"ottingen, Germany}
\email{joerg.bruedern@mathematik.uni-goettingen.de}
\author[Trevor D. Wooley]{Trevor D. Wooley}
\address{Department of Mathematics, Purdue University, 150 N. University Street, West 
Lafayette, IN 47907-2067, USA}
\email{twooley@purdue.edu}
\date{}
\dedicatory{}
\begin{abstract}We provide estimates for $s^{\rm th}$ moments of biquadratic smooth 
Weyl sums, when $10\le s\le 12$, by enhancing the second author's iterative method that 
delivers estimates beyond the classical convexity barrier. As a consequence, all sufficiently 
large integers $n$ satisfying $n\equiv r\mmod{16}$, with $1\le r\le 12$, can be written as 
a sum of $12$ biquadrates of smooth numbers.
\end{abstract}
\maketitle

\section{Introduction} Our focus in this memoir lies on the moments of quartic smooth 
Weyl sums
\[
g(\alp;P,R)=\sum_{x\in \calA(P,R)}e(\alp x^4),
\]
where $e(z)=e^{2\pi iz}$ and $\calA(P,R)$ denotes the set of numbers $n\in [1,P]$, all of 
whose prime divisors are at most $R$. In this paper, we refer to the number $\Del_t$ as 
an {\it admissible exponent} for the positive real number $t$ if there exists a positive 
number $\eta$ such that, whenever $1\le R\le P^\eta$, one has
\[
\int_0^1|g(\alp;P,R)|^t\d\alp \ll P^{t-4+\Del_t}.
\]
Recent work \cite[Theorem 1.3]{BW2018} of the authors shows that 
$\Del_{10}=0.1991466$ is an admissible exponent. It is implicit in work of Vaughan 
\cite[Lemma 5.2]{Vau1989a}, moreover, that the exponent $\Del_{12}=0$ is admissible. 
Hitherto, the sharpest upper bounds available for admissible exponents $\Del_t$ in the 
range $10\le t\le 12$ stem from linear interpolation, via H\"older's inequality, between the 
$10^{\rm th}$ and $12^{\rm th}$ moments. Our principal goal in this paper is to derive 
estimates going beyond this classical convexity barrier. In particular, we seek to establish 
the existence of a number $t_0$, with $t_0<12$, having the property that the exponent 
$\Del_{t_0}=0$ is admissible. It transpires that the existence of such a number $t_0$ has 
attractive consequences for additive problems involving biquadrates.\par

A complete description of our new admissible exponents would be cumbersome to report at 
this stage, so we defer a full account to \S\S4 and 5. An indication of the kind of results 
available is provided in the following theorem.

\begin{theorem}\label{theorem1.1} The exponents $\Del_t$ presented in Table 
\ref{table1} are all admissible.
\end{theorem}

The exponents in Table \ref{table1} are all rounded up in the final decimal place presented. 
A more precise determination of the number $t_0$ to which we alluded above is given in 
our second theorem.

\begin{table}\label{table1}
    \begin{tabular}{ | l | l | l | l | l | l | l | p{5cm} |}
    \hline
    \ $t$ & \ \ \ 10.00& \ \ \ 11.00& \ \ \ 11.50& \ \ \ 11.75& \ \ \ 11.96& \ \ \ 12.00\\ \hline
    $\Del_t$&0.1991466&0.0806719&0.0323341&0.0128731&0.0000000&0.0000000\\ \hline
    \end{tabular}
\vskip.2cm
\caption{Admissible exponents for $10\le t\le 12$.}
\end{table}

\begin{theorem}\label{theorem1.2} Whenever $t\ge 11.95597$, the exponent 
$\Del_t=0$ is admissible. Thus, there exists a positive number $\eta$ such that, when 
$1\le R\le P^\eta$, one has
\begin{equation}\label{1.1}
\int_0^1|g(\alp;P,R)|^t\d\alp \ll P^{t-4}.
\end{equation}
\end{theorem} 

The upper bound (\ref{1.1}) presented in Theorem \ref{theorem1.2} is the first established 
in which a moment of a biquadratic smooth Weyl sum beyond the $4^{\rm th}$ but smaller 
than the $12^{\rm th}$ has the conjectured order of magnitude. As experts will recognise, 
such a mean value offers the prospect of establishing results of Waring-type concerning 
sums of $12$ or more smooth biquadrates. In this context, we shall refer to a positive 
integer $n$ as being {\it $R$-smooth} when all of its prime divisors are no larger than $R$. 
We record the following consequence of Theorem \ref{theorem1.2}.

\begin{corollary}\label{corollary1.3} There exists a positive number $\kap$ having the 
property that every sufficiently large integer $n$ satisfying $n\equiv r\mmod{16}$, with 
$1\le r\le 12$, can be written as the sum of $12$ biquadrates of $(\log n)^\kap$-smooth 
integers.
\end{corollary}

Recall that, since for all integers $m$ one has $m^4\in \{0,1\}\mmod{16}$, then 
whenever $n$ is the sum of $12$ biquadrates, it follows that $n\equiv r\mmod{16}$ for 
some integer $r$ with $0\le r\le 12$. Moreover, the integer $31\cdot 16^s$ $(s\ge 0)$ is 
never the sum of $12$ biquadrates. The condition on $r$ in Corollary \ref{corollary1.3} is 
therefore implied by local solubility considerations.\par

An earlier conclusion of Harcos \cite{Har1997} delivers a conclusion similar to that of 
Corollary \ref{corollary1.3} for sums of $17$ biquadrates, though with smoothness 
parameter $(\log n)^\kap$ replaced by $\exp\left( c(\log n\log \log n)^{1/2}\right)$, for a 
suitable positive constant $c$. By adapting the treatment of \cite[\S5]{BW2001}, 
concerning Waring's problem for cubes of smooth numbers, to the present setting, it would 
be routine using Theorem \ref{theorem1.2} to establish a version of Corollary 
\ref{corollary1.3} for sums of $12$ biquadrates of 
$\exp\left( c(\log n\log \log n)^{1/2}\right)$-smooth integers. The reduction of the 
smoothness parameter to $(\log n)^\kap$ is made possible by recent work of Drappeau 
and Shao \cite{DS2016}. Once equipped with the estimate (\ref{1.1}), the details of the 
proof of Corollary \ref{corollary1.3} are a routine, though not especially brief, modification 
of the argument of \cite{DS2016}. Since this is hardly the main point of the present memoir, 
we eschew any account of the proof of Corollary \ref{corollary1.3}, leaving 
the reader to follow the pedestrian walkway already provided in \cite{DS2016}.\par

A second application of Theorem \ref{theorem1.2} concerns the solubility of pairs of 
diagonal quartic equations of the shape
\begin{align*}
a_1x_1^4+\ldots +a_sx_s^4&=0\\
b_1x_1^4+\ldots +b_sx_s^4&=0,
\end{align*}
wherein $a_i,b_i\in \dbZ$ are fixed with $(a_i,b_i)\ne (0,0)$ for $1\le i\le s$. Suppose that 
$s\ge 22$ and that in any diagonal quartic form lying in the pencil of the two forms defining 
these equations, there are at least $12$ variables having non-zero coefficients. The authors 
show in \cite{BW2021} that, provided this system has non-singular real and $p$-adic 
solutions for each prime number $p$, then it possesses $\calN(P)\gg P^{s-8}$ integral 
solutions $\bfx$ with $|x_i|\le P$ $(1\le i\le s)$. This conclusion improves on an earlier one 
\cite{BW2015} of the authors in which the condition on the pencil insists that at least $s-7$ 
variables have non-zero coefficients. An important ingredient in the proof of this new result 
is an optimal upper bound of the shape (\ref{1.1}) for some $t<12$, as provided by 
Theorem \ref{theorem1.2}.\par

We establish Theorems \ref{theorem1.1} and \ref{theorem1.2} by applying estimates for 
the mean values
 \begin{equation}\label{1.2}
U_s(P,R)=\int_0^1|g(\alp;P,R)|^s\d\alp ,
\end{equation}
with various values of $s\in [4,12]$. Seminal work of Vaughan \cite[Theorem 4.3]
{Vau1989a} derives useful admissible exponents when $s\in \{6,8,10\}$ and shows also, 
implicitly, that the exponent $\Del_{12}=0$ is admissible. Following some refinement in 
these exponents in subsequent work of Vaughan \cite[Theorem 1.3]{Vau1989b}, the 
second author introduced a new approach \cite{Woo1995} in which moments of fractional 
order can be estimated in a manner more efficient than mere application of H\"older's 
inequality to interpolate between admissible exponents available for even values of $s$. 
This tool was fully exploited in work \cite[Theorem 2 and page 393]{BW2000} of the 
authors. Despite recent progress on the $10^{\rm th}$ moment (see 
\cite[Theorem 1.3]{BW2018}), the sharpest upper bounds hitherto available for admissible 
exponents $\Del_s$ in the range $10\le s\le 12$ stem from linear interpolation, via 
H\"older's inequality, between $10^{\rm th}$ and $12^{\rm th}$ moments.\par

A pedestrian application of the iterative method of \cite{Woo1995} would seek to break the 
classical convexity barrier, between $10^{\rm th}$ and $12^{\rm th}$ moments, by 
applying an $8^{\rm th}$ moment of an auxiliary exponential sum of the shape
\begin{equation}\label{1.3}
\Ftil_1(\alp)=\sum_{\substack{u\in \calA(P^\phi R,R)\\ u>P^\phi}}
\sum_{\substack{z_1,z_2\in \calA(P,R)\\ z_1\equiv z_2\mmod{u^4}\\ z_1\ne z_2}}
e\left(\alp u^{-4}(z_1^4-z_2^4)\right).
\end{equation}
Here, the parameter $\phi$ is chosen appropriately in the range $0\le \phi\le 1/4$. It 
transpires that this approach bounds the mean value $U_s(P,R)$ defined in (\ref{1.2}) in 
terms of corresponding bounds for $U_{s-2}(P,R)$ and $U_t(P,R)$, wherein $t$ is a 
parameter to be chosen with $\frac{8}{7}(s-2)\le t\le \frac{4}{3}(s-2)$. This, it turns out, 
is too inefficient to be useful. What makes the exponential sum awkward to handle is the 
constraint that $z_1$ and $z_2$ both be smooth. Drawing inspiration from an argument 
presented in \cite[\S\S1 and 3]{Woo2015}, in this paper we estimate the auxiliary integral
\[
\int_0^1\Ftil_1(\alp)|g(\alp;P^{1-\phi},R)|^{s-2}\d\alp
\]
in terms of the mediating mean value
\begin{equation}\label{1.4}
\int_0^1|\Ftil_1(\alp)^2g(\alp;P^{1-\phi},R)^8|\d\alp .
\end{equation}
By orthogonality, this mean value counts the number of solutions of an underlying 
Diophantine equation. This observation permits us the expedient step of discarding the 
constraint that $z_1$ and $z_2$ be smooth in the exponential sum $\Ftil_1(\alp)$ defined 
in (\ref{1.3}), and in this way a useful bound may be derived. One may then introduce the 
array of tools developed by previous scholars to handle classical analogues of 
$\Ftil_1(\alp)$.\par

The approach outlined above succeeds in bounding $U_s(P,R)$ in terms of $U_t(P,R)$ and 
the mean value (\ref{1.4}), with $t=2s-12$. Since $\Ftil_1(\alp)$ may be thought of roughly 
as having the weight of two smooth Weyl sums, the mean value (\ref{1.4}) behaves 
approximately as a $12$-th moment. Yet, with the smoothness constraint discarded, we are 
able to obtain an optimal upper bound for this mean value. It is the latter that permits our 
efficient application of ideas from the machinery associated with breaking classical convexity. 
A careful analysis of these ideas would show, in fact, that admissible exponents 
$\Del_{12-u}$ exist satisfying $\Del_{12-u}\ll u^\bet$, for small values of $u$, wherein
\[
\bet>\frac{\log (38/15)}{\log 2}=1.341\ldots .
\]
Since the approach of $\Del_{12-u}$ towards $0$ as $u\rightarrow 0$ is more rapid than 
linear decay, we may apply a Weyl-type estimate to establish the existence of a positive 
number $u_0$ for which $\Del_{12-u_0}=0$ is admissible. Appeal to the Keil-Zhao device, 
just as in \cite[\S6]{Woo2015}, delivers the sharpest conclusions presently available 
concerning upper bounds for permissible values of $u_0$.\par

This memoir is organised as follows. We begin in \S2 by deriving the auxiliary mean value 
estimate associated with (\ref{1.4}). Then, in \S3, we employ this mean value within the 
infrastructure permeating the theory of breaking convexity so as to derive the mean value 
estimates required in deriving new admissible exponents. The iterative relations delivering 
these new exponents are derived in \S4, and numerical values follow. These results 
establish all of the admissible exponents asserted in Theorem \ref{theorem1.1} save for the 
claim that $\Del_{11.96}=0$. In \S5 we discuss the Keil-Zhao device and its implications for 
mean values of biquadratic smooth Weyl sums. The final assertion of Theorem 
\ref{theorem1.1} follows, as does its more precise analogue recorded in Theorem 
\ref{theorem1.2}.\par

In this paper, we adopt the convention that whenever $\eps$, $P$ or $R$ appear in a 
statement, either implicitly or explicitly, then for each $\eps>0$, there exists a positive 
number $\eta=\eta(\eps)$ such that the statement holds whenever $R\le P^\eta$ and $P$ 
is sufficiently large in terms of $\eps$ and $\eta$. Implicit constants in Vinogradov's notation 
$\ll$ and $\gg$ will depend at most on $\eps$ and $\eta$. Since our iterative methods 
involve only a finite number of statements (depending at most on $\eps$), there is no 
danger of losing control of implicit constants. Finally, we write 
$\|\tet\|=\min_{y\in\dbZ}|\tet-y|$.\par
 
\noindent {\bf Acknowledgements:} The authors acknowledge support by Akademie der 
Wissenschaften zu G\"ottingen and Deutsche Forschungsgemeinschaft Project Number 
255083470. The second author's work is supported by the NSF grants DMS-1854398 and 
DMS-2001549. 

\section{The auxiliary mean value estimate} We begin by recalling some upper bounds on 
admissible exponents.

\begin{lemma}\label{lemma2.1} The exponents
\[
\Del_8=0.594193,\quad \Del_{10}=0.1991466\quad \text{and}\quad \Del_{12}=0
\]
are admissible.
\end{lemma}

\begin{proof} The conclusion concerning $\Del_8$ follows from \cite[Theorem 2]{BW2000} 
and the discussion surrounding the table of exponents on \cite[page 393]{BW2000}. The 
assertion concerning $\Del_{10}$ is established in \cite[Theorem 1.3]{BW2018}. Finally, the 
validity of the admissible exponent $\Del_{12}=0$ is a consequence of 
\cite[Lemma 5.2]{Vau1989a}.
\end{proof}

In advance of the introduction of the mean 
value estimate central to our subsequent deliberations, we must introduce some notation. 
Let $\phi$ be a real number with $0\le \phi\le 1/4$, and write
\begin{equation}\label{2.1}
M=P^\phi,\quad H=PM^{-4}\quad \text{and}\quad Q=PM^{-1}.
\end{equation}
We define the difference polynomial
\[
\Psi(z,h,m)=m^{-4}\left( (z+hm^4)^4-(z-hm^4)^4\right) =8hz(z^2+h^2m^8),
\]
and then introduce the exponential sum having argument $\Psi(z,h,m)$, namely
\begin{equation}\label{2.2}
F_1(\alp)=\sum_{1\le h\le H}\sum_{M<m\le MR}\sum_{1\le z\le 2P}e\left( 8\alp 
hz(z^2+h^2m^8)\right) .
\end{equation}
It is convenient to work with the exponential sum $g_\flat(\alp)=g(\alp;2Q,R)$.

The remainder of this section is devoted to the estimation of the auxiliary mean value
\begin{equation}\label{2.3}
T=\int_0^1|F_1(\alp)^2g_\flat(\alp)^8|\d\alp .
\end{equation}

\begin{lemma}\label{lemma2.2} One has
\[
T\ll P^\eps (PHM)^2Q^4\left( 1+(PM^{-6})^{1/10}\right) .
\]
\end{lemma}

\begin{proof} We follow closely certain aspects of the argument of the proof of 
\cite[Lemma 2.4]{BW2018} associated with the corresponding analysis therein of the mean 
value defined by \cite[equation (2.2)]{BW2018}. In this way, writing $\calB(l)$ for the set of 
all integers $z$ with $1\le z\pm l\le 4P$ and $z\equiv l\mmod{2}$, we find by applying 
Cauchy's inequality that
\begin{equation}\label{2.4}
|F_1(\alp)|^2\ll P^{1+\eps}H^2M^2+P^\eps HM \left( D(\alp)E(\alp)\right)^{1/2},
\end{equation}
in which
\[
D(\alp)=\sum_{1\le h\le H}\sum_{1\le l\le 2P}\Biggl| \sum_{z\in \calB(l)}e(6\alp hlz^2)
\Biggr|^2
\]
and
\[
E(\alp)=\sum_{1\le h\le H}\sum_{1\le l\le 2P}\Biggl| \sum_{M<m\le MR}e(8\alp lh^3m^8)
\Biggr|^2 .
\]
The trivial estimate $E(\alp)\ll PH(MR)^2$ combines with (\ref{2.3}) and (\ref{2.4}) to give
\begin{equation}\label{2.5}
T\ll P^{1+\eps}H^2M^2T_1+P^\eps (PH^3M^4)^{1/2}T_2,
\end{equation}
where
\begin{equation}\label{2.6}
T_1=\int_0^1|g_\flat (\alp)|^8\d\alp \quad \text{and}\quad T_2
=\int_0^1D(\alp)^{1/2}|g_\flat(\alp)|^8\d\alp .
\end{equation}

\par We estimate $T_2$ via the Hardy-Littlewood method. Given integers $a$ and $q$ with 
$0\le a\le q\le P$ and $(a,q)=1$, let $\grP(q,a)$ denote the set of all $\alp\in [0,1)$ with 
$|q\alp -a|\le PQ^{-4}$, and let $\grP$ denote the union of these intervals. Note that this 
union is disjoint. Define the function $\Phi:[0,1)\rightarrow [0,1]$ by putting
\[
\Phi(\alp)=(q+Q^4|q\alp -a|)^{-1},
\]
when $\alp \in \grP(q,a)\subseteq \grP$, and put $\Phi(\alp)=0$ when $\alp \not\in \grP$. 
Having introduced essentially the same notation here as that employed in the proof of 
\cite[Lemma 2.4]{BW2018}, we find that when $\alp \in [0,1)$, the proof of 
\cite[Lemma 3.1]{Vau1989a} shows that
\[
D(\alp)\ll P^{2+\eps}H+P^{3+\eps}H\Phi (\alp).
\]
Write
\[
T_3=\int_\grP \Phi(\alp)^{1/2}|g_\flat(\alp)|^8\d\alp .
\]
Then we deduce from (\ref{2.5}) and (\ref{2.6}) that
\begin{align}
T&\ll P^{1+\eps}H^2M^2T_1+P^\eps(P^3H^4M^4)^{1/2}T_1+P^\eps
(P^4H^4M^4)^{1/2}T_3,\notag \\
&\ll P^{3/2+\eps}H^2M^2T_1+P^{2+\eps}H^2M^2T_3.\label{2.7}
\end{align}

\par We see from Lemma \ref{lemma2.1} that there is an admissible exponent $\Del_8$ 
smaller than $3/5$, and thus $T_1\ll Q^{23/5}$. Consequently, it follows from (\ref{2.1}) 
that
\begin{equation}\label{2.8}
P^{-1/2}T_1\ll P^{-1/2}Q^{23/5}=Q^4(PM^{-6})^{1/10}.
\end{equation}
Meanwhile, an application of Schwarz's inequality reveals that
\[
T_3\le \biggl( \int_\grP \Phi (\alp)|g_\flat(\alp)|^4\d\alp \biggr)^{1/2} 
\biggl( \int_0^1|g_\flat (\alp)|^{12}\d\alp \biggr)^{1/2}.
\]
It follows from \cite[Lemma 2]{Bru1988} that
\[
\int_\grP \Phi (\alp)|g_\flat(\alp)|^4\d\alp \ll Q^{\eps-4}(PQ^2+Q^4)\ll Q^\eps .
\]
On the other hand, by Lemma \ref{lemma2.1} we have
\[
\int_0^1|g_\flat(\alp)|^{12}\d\alp \ll Q^8.
\]
We thus conclude that
\[
T_3\ll (Q^\eps)^{1/2}(Q^8)^{1/2}\ll Q^{4+\eps}.
\]
On substituting this estimate together with (\ref{2.8}) into (\ref{2.7}), we infer that
\[
T\ll P^\eps (PHM)^2Q^4(PM^{-6})^{1/10}+P^\eps (PHM)^2Q^4.
\]
This completes the proof of the lemma.
\end{proof}

\section{Mean values associated with breaking convexity}
The auxiliary mean value $T$ defined in (\ref{2.3}) captures the essentials of what is 
needed in our application of the second author's work \cite{Woo1995} on breaking 
convexity, but not the details. We must therefore expend further effort in order that the 
intricacies of our full argument be accommodated. We begin with some additional notation. 
We define the modified set of smooth numbers $\calB(L,\pi,R)$ for prime numbers $\pi$ by 
putting
\[
\calB(L,\pi,R)=\{ n\in \calA(L\pi,R):\text{$n>L$, $\pi|n$, and $\pi'|n$ implies that 
$\pi'\ge \pi$}\}.
\]
In this definition we use $\pi'$ to denote a prime number. We note that this definition 
corrects the analogous definition in the preamble to \cite[equation (3.1)]{Woo2015}. 
Recalling the notation (\ref{2.1}), we put
\begin{equation}\label{3.1}
\Ftil_{d,e}(\alp;\pi)=\sum_{u\in \calB(M/d,\pi,R)}\sum_{\substack{x,y\in \calA(P/(de),R)
\\ (x,u)=(y,u)=1\\ x\equiv y\mmod{u^4}\\ y<x}}e\left(\alp u^{-4}(x^4-y^4)\right) ,
\end{equation}
\begin{equation}\label{3.2}
F_{d,e}(\alp)=\sum_{1\le z\le 2P/(de)}\sum_{1\le h\le Hd^3/e}\sum_{M/d<u\le MR/d}
e\left( 8\alp hz(z^2+h^2u^8)\right)
\end{equation}
and
\begin{equation}\label{3.3}
\ftil(\alp;P,M,R)=\max_{m>M}\biggl| \sum_{x\in \calA(P/m,R)}e(\alp x^4)\biggr| .
\end{equation}
We note that $F_{d,e}(\alp)=0$ when $e>Hd^3$. Finally, we put
\begin{equation}\label{3.4}
\Ups_{d,e,\pi}(P,R;\phi)=\int_0^1|\Ftil_{d,e}(\alp;\pi)^2\ftil(\alp;P/(de),M/d,\pi)^8|\d\alp .
\end{equation}

\par Our initial step is to bound $\Ups_{d,e,\pi}(P,R;\phi)$ in terms of a similar mean value 
in which $F_{d,e}(\alp)$ is substituted for $\Ftil_{d,e}(\alp;\pi)$.

\begin{lemma}\label{lemma3.1} When $\pi\le R$, one has
\[
\Ups_{d,e,\pi}(P,R;\phi)\ll P^\eps \int_0^1|F_{d,e}(\alp)^2g(\alp;2Q/e,R)^8|\d\alp .
\]
\end{lemma}

\begin{proof} As in a similar treatment offered during the proof of 
\cite[Lemma 3.1]{Woo2015}, the maximal property of the sum $\ftil(\alp;P/(de),M/d,\pi)$ is 
readily eliminated by application of a standard argument employing the Dirichlet kernel. 
Define
\[
\calD_K(\tet)=\sum_{|m|\le K^4}e(m\tet)\quad \text{and}\quad \calD_K^*(\tet)=\min 
\{ 2K^4+1,\|\tet\|^{-1}\} .
\]
Then for $K\ge 1$ one has the familiar estimate
\begin{equation}\label{3.5}
\int_0^1\calD_K^*(\tet)\d\tet \ll \log (2K).
\end{equation}
Recalling (\ref{2.1}) once more, we see that whenever $m>M$, one has
\[
\sum_{x\in \calA(P/m,R)}e(\alp x^4)=\int_0^1g(\alp+\tet;Q,R)\calD_{P/m}(\tet)\d\tet .
\]
When $m>M$, we have $\calD_{P/m}(\tet)\ll \calD_{P/m}^*(\tet)\le \calD_Q^*(\tet)$, 
and so it follows from (\ref{3.3}) that
\begin{equation}\label{3.6}
\ftil(\alp;P/(de),M/d,\pi)\ll \int_0^1|g(\alp+\tet;Q/e,\pi)|\calD_Q^*(\tet)\d\tet .
\end{equation}

\par We substitute eight copies of (\ref{3.6}) into (\ref{3.4}), deducing that
\[
\Ups_{d,e,\pi}(P,R;\phi)\ll \int_0^1\int_{[0,1)^8}|\Ftil_{d,e}(\alp;\pi)|^2\left( 
\prod_{i=1}^8|g(\alp+\tet_i;Q/e,\pi)|\calD_Q^*(\tet_i)\right) \d\bftet \d\alp .
\]
We next put
\begin{equation}\label{3.7}
\Xi_{d,e,\pi}(\tet)=\int_0^1|\Ftil_{d,e}(\alp;\pi)^2g(\alp+\tet;Q/e,\pi)^8|\d\alp .
\end{equation}
Then by applying the elementary bound $|z_1\cdots z_8|\le |z_1|^8+\ldots +|z_8|^8$, and 
invoking symmetry, we discern via (\ref{3.5}) that
\begin{align}
\Ups_{d,e,\pi}(P,R;\phi)&\ll \biggl( \int_0^1 \Xi_{d,e,\pi}(\tet_1)
\calD_Q^*(\tet_1)\d\tet_1\biggr) \prod_{i=2}^8 \int_0^1\calD_Q^*(\tet_i)\d\tet_i\notag \\
&\ll Q^\eps \int_0^1 \Xi_{d,e,\pi}(\tet)\calD_Q^*(\tet)\d\tet .\label{3.8}
\end{align}

\par We relate $\Xi_{d,e,\pi}(\tet)$ to the number of integral solutions of the equation
\begin{equation}\label{3.9}
u_1^{-4}(x_1^4-y_1^4)-u_2^{-4}(x_2^4-y_2^4)=\sum_{j=1}^4(w_{2j-1}^4-w_{2j}^4),
\end{equation}
wherein, for $i=1$ and $2$, one has the constraints
\[
u_i\in \calB(M/d,\pi,R),\quad x_i,y_i\in \calA(P/(de),R),
\]
\[
(x_i,u_i)=(y_i,u_i)=1,\quad x_i\equiv y_i\mmod{u_i^4}\quad \text{and}\quad y_i<x_i,
\]
and in addition $w_j\in \calA(Q/e,\pi)$ $(1\le j\le 8)$. Indeed, by orthogonality, it follows 
from (\ref{3.1}) and (\ref{3.7}) that $\Xi_{d,e,\pi}(\tet)$ counts the number of these 
solutions, with each solution counted with weight
\[
e\biggl( -\tet\sum_{j=1}^4(w_{2j-1}^4-w_{2j}^4)\biggr) .
\]
Since this weight is unimodular, we find that $|\Xi_{d,e,\pi}(\tet)|$ is bounded above by the 
corresponding unweighted count of solutions, and hence by the number of integral solutions 
of the equation (\ref{3.9}) with the constraints, for $i=1$ and $2$,
\[
M/d<u_i\le MR/d,\quad 1\le y_i<x_i\le P/(de),\quad x_i\equiv y_i\mmod{u_i^4},
\]
and in addition $w_j\in \calA(Q/e,R)$ $(1\le j\le 8)$.\par

Next we substitute $z_i=x_i+y_i$ and $h_i=(x_i-y_i)u_i^{-4}$ $(i=1,2)$. Then we see from 
the conditions on $x_i$ and $y_i$ that $1\le h_i\le (P/(de))(M/d)^{-4}$ $(i=1,2)$. Moreover, 
one has
\[
2x_i=z_i+h_iu_i^4\quad \text{and}\quad 2y_i=z_i-h_iu_i^4\quad (i=1,2).
\]
Then since
\[
u^{-4}\left( (z+hu^4)^4-(z-hu^4)^4\right) =8hz(z^2+h^2u^8),
\]
we deduce via (\ref{2.1}) that $|\Xi_{d,e,\pi}(\tet)|$ is bounded above by the number of 
integral solutions of the equation
\[
8h_1z_1(z_1^2+h_1^2u_1^8)-8h_2z_2(z_2^2+h_2^2u_2^8)=
\sum_{j=1}^4(w_{2j-1}^4-w_{2j}^4),
\]
in which, for $i=1$ and $2$, one has
\[
M/d<u_i\le MR/d,\quad 1\le z_i\le 2P/(de)\quad \text{and}\quad 1\le h_i\le Hd^3/e,
\]
and in addition $w_j\in \calA(2Q/e,R)$ $(1\le j\le 8)$.\par

We may now recall (\ref{3.2}) and invoke orthogonality to obtain the upper bound
\[
|\Xi_{d,e,\pi}(\tet)|\le \int_0^1 |F_{d,e}(\alp)^2g(\alp;2Q/e,R)^8|\d\alp .
\]
By substituting this upper bound into (\ref{3.8}) and recalling (\ref{3.5}), we thus conclude 
that
\begin{align*}
\Ups_{d,e,\pi}(P,R;\phi)&\ll Q^\eps \biggl( \int_0^1 \calD_Q^*(\tet)\d\tet \biggr) 
\int_0^1 |F_{d,e}(\alp)^2g(\alp;2Q/e,R)^8|\d\alp \\
&\ll Q^{2\eps}\int_0^1 |F_{d,e}(\alp)^2g(\alp;2Q/e,R)^8|\d\alp .
\end{align*}
This completes the proof of the lemma.
\end{proof}

By applying Lemma \ref{lemma3.1}, we relate $\Ups_{d,e,\pi}(P,R;\phi)$ to the mean value 
$T$ defined in (\ref{2.3}), and bounded in Lemma \ref{lemma2.2}.

\begin{lemma}\label{lemma3.2} Suppose that
\[
\pi\le R,\quad 1\le d\le M,\quad 1\le e\le \min\{Q,Hd^3\}\quad \text{and}\quad 1/6\le \phi
\le 1/4.
\]
Then
\[
\Ups_{d,e,\pi}(P,R;\phi)\ll P^\eps (PHM)^2Q^4d^{5/2}e^{-8}.
\]
\end{lemma}

\begin{proof} On recalling (\ref{2.2}) and (\ref{3.2}), we find from Lemma \ref{lemma2.2} 
that
\[
\int_0^1|F_{1,1}(\alp)^2g(\alp;2Q,R)^8|\d\alp \ll P^\eps 
(PHM)^2Q^4\left( 1+(PM^{-6})^{1/10}\right) .
\]
We apply this estimate with $P/(de)$ in place of $P$, and with $M/d$ in place of $M$. In 
alignment with (\ref{2.1}), we then have also $Hd^3/e$ in place of $H$, and $Q/e$ in place 
of $Q$. The hypotheses of the lemma concerning $e$ and $\phi$ ensure that
\[
(M/d)^4\left(P/(de)\right)^{-1}=e/(Hd^3)\le 1,
\]
whence $(M/d)^4\le P/(de)$, as well as
\[
\left(P/(de)\right)(M/d)^{-6}=(PM^{-6})d^5e^{-1}\le d^5.
\]
Hence we obtain the bound
\[
\int_0^1|F_{d,e}(\alp)^2g(\alp;2Q/e,R)^8|\d\alp \ll P^\eps \left( \frac{P}{de}\cdot 
\frac{Hd^3}{e} \cdot \frac{M}{d}\right)^2\left( \frac{Q}{e}\right)^4
\left( 1+(d^5)^{1/10}\right) .
\]
This bound applied in concert with Lemma \ref{lemma3.1} delivers the conclusion of the 
lemma.
\end{proof}

Finally, we recall an estimate for the mean value
\begin{equation}\label{3.10}
\Util_s(P,M,R)=\int_0^1 \ftil(\alp;P,M,R)^s\d\alp .
\end{equation}

\begin{lemma}\label{lemma3.3} Suppose that $s>1$ and that $\Del_s$ is an admissible 
exponent. Then whenever $P>M$ and $R>2$, one has 
$\Util_s(P,M,R)\ll_s (P/M)^{s-4+\Del_s+\eps}$.
\end{lemma}

\begin{proof} This is immediate from \cite[Lemma 3.2]{Woo1995}, on noting the 
definition of an admissible exponent $\Del_s$ used within this paper.
\end{proof}

\section{New admissible exponents for $s>10$} The mean value estimates of \S\S2 and 3 
may be converted into admissible exponents by utilising the machinery of 
\cite[\S\S2-4]{Woo1995}. In this context, we write
\begin{equation}\label{4.1}
\Ome_{d,e,\pi}(P,R;\phi)=\int_0^1|\Ftil_{d,e}(\alp;\pi)\ftil(\alp;P/(de),M/d,\pi)^{s-2}|
\d\alp ,
\end{equation}
and then define
\begin{equation}\label{4.2}
\calU_s(P,R)=\sum_{1\le d\le D}\sum_{\pi\le R}\sum_{1\le e\le Q}d^{2-s/2}e^{s/2-1}
\Ome_{d,e,\pi}(P,R;\phi).
\end{equation}
The key lemma for our present deliberations is the following.

\begin{lemma}\label{lemma4.1} Suppose that $s>4$ and $0<\phi\le 1/4$. Suppose also 
that $\Del_s$ and $\Del_{s-2}$ are admissible exponents, and put $\mu_t=t-4+\Del_t$ 
$(t=s-2,s)$. Then whenever $1\le D\le P^{1/4}$, one has
\[
U_s(P,R)\ll P^{\mu_s+\eps}D^{s/2-\mu_s}+MP^{1+\mu_{s-2}+\eps}+P^{\left( \frac
{s-3}{s-2}\right) \mu_s+\eps}V_s(P,R),
\]
where
\[
V_s(P,R)=\left( PM^{s-2}Q^{\mu_{s-2}}+M^{s-3}\calU_s(P,R)\right)^{1/(s-2)}.
\]
\end{lemma}

\begin{proof} On noting the definition of an admissible exponent, the stated conclusion is 
immediate on substituting the conclusion of \cite[Lemma 3.3]{Woo1995} into that of 
\cite[Lemma 2.3]{Woo1995}.
\end{proof}

We may now announce our new admissible exponents.

\begin{lemma}\label{lemma4.2}
Let $u$ be a real number with $0\le u\le 2$. Suppose that the exponents $\Del_{10-u}$ 
and $\Del_{12-2u}$ are both admissible and satisfy
\begin{equation}\label{4.3}
2\Del_{10-u}-\tfrac{4}{5}\le \Del_{12-2u}\le 2\Del_{10-u}.
\end{equation}
Put
\[
\Del_{12-u}^*=\frac{3\Del_{12-2u}}{8-2\Del_{10-u}+\Del_{12-2u}}.
\]
Then whenever $\Del_{12-u}>\Del_{12-u}^*$, the exponent $\Del_{12-u}$ is admissible.
\end{lemma}

\begin{proof} We initiate our discussion by estimating the mean value 
$\Ome_{d,e,\pi}(P,R;\phi)$. Here and throughout the proof, we set $s=12-u$. Suppose 
that
\[
d\le M,\quad e\le Q,\quad \pi\le R\quad \text{and}\quad 1/6\le \phi\le 1/4.
\]
Then on recalling (\ref{3.4}) and (\ref{3.10}), an application of Schwarz's inequality to 
(\ref{4.1}) reveals that
\begin{equation}\label{4.4}
\Ome_{d,e,\pi}(P,R;\phi)\le \left( \Ups_{d,e,\pi}(P,R;\phi)\right)^{1/2}
\left( \Util_{2s-12}(P/(de),M/d,\pi)\right)^{1/2}.
\end{equation}
Observe that since $0\le u\le 2$ and $\Del_{12-2u}\ge 0$, we have 
$2u-\Del_{12-2u}\le 4$. Then since $2s-12=12-2u$, we deduce from Lemmata 
\ref{lemma3.2} and \ref{lemma3.3} that when $e\le Hd^3$, one has 
\begin{align}
\Ome_{d,e,\pi}(P,R;\phi)&\ll P^\eps \left( (PHM)^2Q^4d^{5/2}e^{-8}\right)^{1/2}
\left( (Q/e)^{8-2u+\Del_{12-2u}}\right)^{1/2}\notag \\
&\ll P^{1+\eps}HMQ^{6-u+\frac{1}{2}\Del_{12-2u}}d^{5/4}e^{-6}.\label{4.5}
\end{align}
When instead $e>Hd^3$, it follows from (\ref{3.2}) that $F_{d,e}(\alp)=0$, and hence 
we deduce from Lemma \ref{lemma3.1} that $\Ups_{d,e,\pi}(P,R;\phi)=0$. In such 
circumstances we infer from (\ref{4.4}) that $\Ome_{d,e,\pi}(P,R;\phi)=0$.\par

Provided that we make a choice of $D$ with $D\le M$, it therefore follows by substituting 
(\ref{4.5}) into (\ref{4.2}) that
\[
\calU_s(P,R)\ll P^{1+\eps}HMQ^{6-u+\frac{1}{2}\Del_{12-2u}}\Sig_0,
\]
where
\[
\Sig_0=\sum_{1\le d\le D}\sum_{\pi\le R}\sum_{1\le e\le \min\{Q,Hd^3\}}
d^{\frac{13}{4}-\frac{s}{2}}e^{-1}.
\]
Thus, on recalling our convention concerning $\eps$ and $R$, we deduce that
\[
\calU_s(P,R)\ll P^{1+\eps}HMQ^{6-u+\frac{1}{2}\Del_{12-2u}}.
\]
In the notation of Lemma \ref{lemma4.1}, we thus obtain the bound
\[
V_s(P,R)^{s-2}\ll P^\eps M^{s-3}(\Psi_1+\Psi_2),
\]
where
\[
\Psi_1=PMQ^{6-u+\Del_{10-u}}\quad \text{and}\quad \Psi_2=PHMQ^{6-u+\frac{1}{2}
\Del_{12-2u}}.
\]

\par By reference to (\ref{2.1}), the equation $\Psi_1=\Psi_2$ implicitly determines a 
linear equation for $\phi$, namely
\[
1+\phi+(6-u+\Del_{10-u})(1-\phi)=2-3\phi+(6-u+\tfrac{1}{2}\Del_{12-2u})(1-\phi) .
\]
This equation has the solution $\phi=\phi_0$, where
\[
\phi_0=\frac{1+\tfrac{1}{2}\Del_{12-2u}-\Del_{10-u}}{4+\tfrac{1}{2}\Del_{12-2u}
-\Del_{10-u}} .
\]
Observe that the hypothesis (\ref{4.3}) ensures that $\phi_0\le 1/4$, and also that
\[
6\phi_0-1=\frac{\frac{5}{2}(\Del_{12-2u}+\frac{4}{5}-2\Del_{10-u})}{\frac{18}{5}+
\frac{1}{2}(\Del_{12-2u}+\frac{4}{5}-2\Del_{10-u})}\ge 0,
\]
whence $\phi_0\ge 1/6$. This justifies our earlier assumption that $1/6\le \phi\le 1/4$. We 
define the exponent $\mu_s$ via the relation
\[
\mu_s=\mu_{s-2}(1-\phi_0)+1+(s-2)\phi_0,
\]
and then put $\Del_s^*=\mu_s+4-s$. Thus we have
\begin{align*}
\Del_s^*&=\Del_{s-2}(1-\phi_0)+4\phi_0-1\\
&=\frac{3\Del_{10-u}+\tfrac{3}{2}\Del_{12-2u}-3\Del_{10-u}}
{4+\tfrac{1}{2}\Del_{12-2u}-\Del_{10-u}}\\
&=\frac{3\Del_{12-2u}}{8+\Del_{12-2u}-2\Del_{10-u}}.
\end{align*}

\par Put $D=P^\ome$, where $\ome$ is any sufficiently small, but fixed, positive number. 
Then we may follow the discussion of \cite[\S4]{Woo1995} so as to confirm via Lemma 
\ref{lemma4.1} that whenever $\Del_{12-2u}$ and $\Del_{10-u}$ are admissible 
exponents, then one has the upper bound
\[
U_s(P,R)\ll P^{\mu_s+\eps},
\] 
whence $\Del_s=\Del_{12-u}$ is also an admissible exponent whenever 
$\Del_{12-u}>\Del_{12-u}^*$. This completes the proof of the lemma.
\end{proof}

Note that in view of Lemma \ref{lemma2.1}, it follows by applying linear interpolation via 
H\"older's inequality that when $0\le u\le 2$, one has
\begin{equation}\label{4.6}
\Del_{12-2u}\le \Del_{10-u}\le 2\Del_{10-u}\le \Del_{12-2u}+\Del_8\le 
\Del_{12-2u}+\tfrac{3}{5}.
\end{equation}
Hence the hypothesis (\ref{4.3}) will always be satisfied in the applications to come.\par

The conclusion of Lemma \ref{lemma4.2} permits the bulk of Theorem \ref{theorem1.1} 
to be established. Since exponents $\Del_s$ admissible throughout the interval 
$10\le s\le 12$ may be of use in future applications, we provide explicit formulae.

\begin{theorem}\label{theorem4.3} Suppose that $0\le t\le 1$. Then the exponent
\[
\Del_{10+t}=0.1991466 - 0.1184747 t
\]
is admissible. In particular, the exponent $\Del_{11}=0.0806719$ is admissible.
\end{theorem}

\begin{proof} Working within the environment (\ref{4.6}), put
\[
\Del_{11}^*=\frac{3\Del_{10}}{8-2\Del_9+\Del_{10}}.
\]
Then, by applying Lemma \ref{lemma4.2} with $u=1$, we find that when $\Del_9$ and 
$\Del_{10}$ are admissible exponents, then the exponent $\Del_{11}$ is admissible 
whenever $\Del_{11}>\Del_{11}^*$. By linear interpolation using Schwarz's inequality, 
we may assume that $\Del_9=\tfrac{1}{2}(\Del_8+\Del_{10})$ is admissible, and thus
\begin{equation}\label{4.7}
\Del_{11}^*\le \frac{3\Del_{10}}{8-\Del_8}.
\end{equation}
But in view of Lemma \ref{lemma2.1}, we may suppose that $\Del_8=0.594193$ and 
$\Del_{10}=0.1991466$. Thus we find from (\ref{4.7}) that 
$\Del_{11}^*\le 0.080671803$, and the final conclusion of the theorem follows.\par

By linear interpolation using H\"older's inequality, it follows from this admissible exponent 
$\Del_{11}$ that when $0\le t\le 1$, the exponent
\[
\Del_{10+t}=(1-t)\Del_{10}+t\Del_{11}
\]
is admissible. The first conclusion of the theorem therefore follows with a modicum of 
computation.
\end{proof}

It might be thought that for values of $t$ with $0<t<1$, a more direct application of 
Lemma \ref{lemma4.2} would yield admissible exponents superior to those obtained in 
Theorem \ref{theorem4.3} via linear interpolation. However, working within the 
environment (\ref{4.6}), put
\[
\Del_{10+t}^*=\frac{3\Del_{8+2t}}{8-2\Del_{8+t}+\Del_{8+2t}}\quad (0\le t\le 1).
\]
Then an application of Lemma \ref{lemma4.2} with $u=2-t$ shows that the exponent 
$\Del_{10+t}$ is admissible whenever $\Del_{10+t}>\Del_{10+t}^*$. Here, by linear 
interpolation using H\"older's inequality, we may suppose that
\[
\Del_{8+2t}\le (1-t)\Del_8+t\Del_{10}\quad \text{and}\quad 2\Del_{8+t}\le 
\Del_8+\Del_{8+2t}.
\]
Thus we deduce that
\[
\Del_{10+t}^*\le \frac{3\Del_8-3t(\Del_8-\Del_{10})}{8-\Del_8}<0.2407002-0.1600283t.
\]
This estimate is inferior to that of Theorem \ref{theorem4.3} in all cases save $t=1$, in 
which situation it matches the conclusion of the theorem.

\begin{theorem}\label{theorem4.4} Suppose that $0\le t\le 1/2$. Then the exponent
\[
\Del_{11+t}=\frac{0.0806719-0.0959852t}{1+0.0213477t}
\]
is admissible. In particular, the exponent $\Del_{11.5}=0.0323341$ is admissible.
\end{theorem}

\begin{proof} Working within the environment (\ref{4.6}), put
\[
\Del_{11.5}^*=\frac{3\Del_{11}}{8+\Del_{11}-2\Del_{9.5}}.
\]
By applying Lemma \ref{lemma4.2} with $u=1/2$, we find that when $\Del_{9.5}$ and 
$\Del_{11}$ are admissible exponents, then so too is $\Del_{11.5}$ whenever 
$\Del_{11.5}>\Del_{11.5}^*$. By linear interpolation using Schwarz's inequality, we have 
$\Del_{9.5}\le \tfrac{1}{4}(\Del_8+3\Del_{10})$, and thus
\[
\Del_{11.5}^*\le \frac{6\Del_{11}}{16+2\Del_{11}-\Del_8-3\Del_{10}}.
\]
On making use of the admissible exponents $\Del_8=0.594193$, $\Del_{10}=0.1991466$ 
and $\Del_{11}=0.0806719$ available from Lemma \ref{lemma2.1} and Theorem 
\ref{theorem4.3}, we thus see that the exponent $\Del_{11.5}=0.0323341$ is admissible.

\par Put
\[
\Del_{11+t}^*=\frac{3\Del_{10+2t}}{8+\Del_{10+2t}-2\Del_{9+t}}\quad (0\le t\le 1/2).
\]
Then, more generally, by applying Lemma \ref{lemma4.2} with $u=1-t$, we find that the 
exponent $\Del_{11+t}$ is admissible whenever $\Del_{11+t}>\Del_{11+t}^*$. Applying 
linear interpolation as before, we find that
\begin{align}
\Del_{11+t}^*&\le \frac{(3-6t)\Del_{10}+6t\Del_{11}}{8+(1-2t)\Del_{10}+2t\Del_{11}-
\Del_8(1-t)-\Del_{10}(1+t)}\notag \\
&=\frac{3\Del_{10}-6t(\Del_{10}-\Del_{11})}
{8-\Del_8+(\Del_8-3\Del_{10}+2\Del_{11})t}.
\label{4.8}
\end{align}
We may suppose that $\Del_8=0.594193$, $\Del_{10}=0.1991466$  and 
$\Del_{11}=0.0806719$, and thus
\[
\frac{3\Del_{10}}{8-\Del_8}<0.0806719,\quad \frac{6(\Del_{10}-\Del_{11})}{8-\Del_8}
>0.0959852
\]
and
\[
\frac{\Del_8-3\Del_{10}+2\Del_{11}}{8-\Del_8}>0.0213477.
\]
Thus we deduce that the upper bound for $\Del_{11+t}$ claimed in the theorem does 
indeed follow from (\ref{4.8}).
\end{proof}

\begin{theorem}\label{theorem4.5} Suppose that $0\le t\le 1/4$. Then the exponent
\[
\Del_{11.5+t}=\frac{0.0323341-0.0769435t}{1+0.0693668t+0.0022534t^2}
\]
is admissible. In particular, the exponent $\Del_{11.75}=0.0128731$ is admissible.
\end{theorem}

\begin{proof} Working within the environment (\ref{4.6}), put
\[
\Del_{11.75}^*=\frac{3\Del_{11.5}}{8+\Del_{11.5}-2\Del_{9.75}}.
\]
Then, by applying Lemma \ref{lemma4.2} with $u=1/4$, we find that when $\Del_{9.75}$ 
and $\Del_{11.5}$ are admissible exponents, then so too is $\Del_{11.75}$ whenever 
$\Del_{11.75}>\Del_{11.75}^*$. By linear interpolation, we have 
$\Del_{9.75}\le \tfrac{1}{8}(\Del_8+7\Del_{10})$, and thus
\[
\Del_{11.75}^*\le \frac{12\Del_{11.5}}{32+4\Del_{11.5}-\Del_8-7\Del_{10}}.
\]
On making use of the admissible exponents $\Del_8=0.594193$, $\Del_{10}=0.1991466$ 
and $\Del_{11.5}=0.0323341$ made available by Lemma \ref{lemma2.1} and Theorem 
\ref{theorem4.4}, we see that $\Del_{11.75}\le 0.0128731$.\par

Put
\[
\Del_{11.5+t}^*=\frac{3\Del_{11+2t}}{8+\Del_{11+2t}-2\Del_{9.5+t}}\quad 
(0\le t\le 1/4).
\]
Then, more generally, by applying Lemma \ref{lemma4.2} with $u=\tfrac{1}{2}-t$, we find 
that the exponent $\Del_{11.5+t}$ is admissible whenever 
$\Del_{11.5+t}>\Del_{11.5+t}^*$. By linear interpolation, we have 
$\Del_{9.5+t}\le \tfrac{1}{4}\left( (1-2t)\Del_8+(3+2t)\Del_{10}\right)$. By substituting this 
estimate together with that supplied by Theorem \ref{theorem4.4} for $\Del_{11+2t}$, we 
obtain the first conclusion of the theorem following a modicum of computation.
\end{proof}

\section{The Keil-Zhao device}
We take a simple approach to the application of the Keil-Zhao device (see \cite[equation 
(3.10)]{Zhao2014} and \cite[page 608]{Keil2014}). This permits estimates more or less half 
the strength of a corresponding minor arc estimate for a classical Weyl sum, though 
applied to smooth Weyl sums. A careful application of the method enables us to apply major 
arc estimates in a manner that avoids any consideration of smooth Weyl sums on minor arcs.

\begin{theorem}\label{theorem5.1} Suppose that $s\ge 8$ and that the exponent $\Del_s$ 
is admissible, and satisfies $\Del_s<1/8$. Suppose also that $u>s+16\Del_s>10$. Then 
\begin{equation}\label{5.1}
\int_0^1|g(\alp;P,R)|^u\d\alp \ll P^{u-4}.
\end{equation}
In particular, the exponent $\Del_w=0$ is admissible for $w\ge u$.
\end{theorem}

\begin{proof} We assume the hypotheses of the statement of the theorem, and define
\begin{equation}\label{5.2}
\del=\tfrac{1}{2}(u-s-16\Del_s).
\end{equation}
Then $\del>0$, and since $s+16\Del_s>10$, it follows that $u>10+2\del$. It is convenient 
throughout to abbreviate $g(\alp;P,R)$ simply to $g(\alp)$. Also, put
\[
I=\int_0^1|g(\alp)|^u\d\alp .
\]
We establish the bound (\ref{5.1}) by means of the Hardy-Littlewood method. Define the set 
of major arcs $\grM$ to be the union of the intervals
\[
\grM(q,a)=\{\alp \in [0,1):|q\alp-a|\le \tfrac{1}{8}P^{-3}\},
\]
with $0\le a\le q\le \tfrac{1}{8}P$ and $(a,q)=1$, and then put 
$\grm=[0,1)\setminus \grM$. Finally, write
\[
G(\alp)=\sum_{1\le x\le P}e(\alp x^4),
\]
and observe that the methods of \cite[Chapter 4]{Vau1997} 
(compare the argument of the proof of \cite[Lemma 5.1]{Vau1989a}) establish that
\begin{equation}\label{5.3}
\int_\grM |G(\alp)|^{5+\del}\d\alp \ll_\del P^{1+\del}.
\end{equation}

\par We introduce auxiliary sets of major and minor arcs in order to transform our mean 
value into one correctly configured for the application of the Keil-Zhao device. Let $\grn$ 
denote the set of real numbers $\alp\in [0,1)$ satisfying
\begin{equation}\label{5.4}
|g(\alp)|\le 2P^{15/16}.
\end{equation}
Then, when $T$ is a real number with $T\ge 1$, denote by $\grN(T)$ the set of real 
numbers $\alp\in [0,1)$ for which
\begin{equation}\label{5.5}
T<|g(\alp)|\le 2T.
\end{equation}
Thus, on writing
\[
\grN=\bigcup_{\substack{j=0\\ 2^j\le P^{1/16}}}^\infty \grN(2^{-j}P),
\]
we see that $\grN\cup \grn=[0,1)$. It follows that
\begin{equation}\label{5.6}
I\le I_0+\sum_{\substack{j=0\\ 2^j\le P^{1/16}}}^\infty I_1(2^{-j}P),
\end{equation}
where
\[
I_0=\int_\grn |g(\alp)|^u\d\alp \quad \text{and}\quad I_1(T)=\int_{\grN(T)}|g(\alp)|^u
\d\alp .
\]

\par The analysis of $I_0$ is direct. In view of (\ref{5.2}) and the bound (\ref{5.4}), 
together with the definition of an admissible exponent, one sees that
\begin{align*}
I_0&\le \biggl( \sup_{\alp \in \grn}|g(\alp)|\biggr)^{16\Del_s+2\del}
\int_0^1|g(\alp)|^s\d\alp \\
&\ll \left( P^{15/16}\right)^{16\Del_s+2\del}P^{s-4+\Del_s}.
\end{align*}
On recalling (\ref{5.2}) once again, we therefore discern that
\begin{equation}\label{5.7}
I_0=o(P^{u-4}).
\end{equation}

\par Consider next any value of $T$ with $P^{15/16}\le T\le P$. We deduce from 
(\ref{5.2}) and (\ref{5.5}) that
\begin{equation}\label{5.8}
I_1(T)\ll T^{16\Del_s-2}\int_{\grN(T)}|g(\alp)|^{u-16\Del_s+2}\d\alp 
=T^{16\Del_s-2}K(T),
\end{equation}
where
\[
K(T)=\int_{\grN(T)}|g(\alp)|^{s+2\del+2}\d\alp .
\]
By Cauchy's inequality, one has
\begin{equation}\label{5.9}
K(T)=\sum_{x,y\in \calA(P,R)}\int_{\grN(T)}|g(\alp)|^{s+2\del}e(\alp (x^4-y^4))\d\alp 
\le PK^*(T)^{1/2},
\end{equation}
where
\begin{align*}
K^*(T)&=\sum_{1\le x,y\le P}\Biggl| \int_{\grN(T)}|g(\alp)|^{s+2\del}e(\alp (x^4-y^4))
\d\alp \Biggr|^2\\
&=\sum_{1\le x,y\le P}\int_{\grN(T)}\int_{\grN(T)}|g(\alp)g(\bet)|^{s+2\del}
e((\alp-\bet)(x^4-y^4))\d\alp \d\bet \\
&=\int_{\grN(T)}\int_{\grN(T)}|g(\alp)g(\bet)|^{s+2\del}|G(\alp-\bet)|^2\d\alp \d\bet .
\end{align*}
Since $[0,1)=\grM\cup \grm$, it follows that
\[
K^*(T)\ll K^*(T;\grM)+K^*(T;\grm),
\]
where, for $\grB\subseteq [0,1)$, we write
\begin{equation}\label{5.10}
K^*(T;\grB)=\underset{\alp-\bet \in \grB}{\int_{\grN(T)}\int_{\grN(T)}}
|g(\alp)g(\bet)|^{s+2\del}|G(\alp-\bet)|^2\d\alp \d\bet .
\end{equation}

\par By applying Weyl's inequality (see \cite[Lemma 2.4]{Vau1997}), one obtains the 
bound
\[
\sup_{\alp-\bet\in \grm}|G(\alp-\bet)|\ll P^{7/8+\eps}.
\]
Thus, invoking symmetry and the trivial estimate (\ref{5.5}) for $|g(\alp)|$ and 
$|g(\bet)|$, one arrives at the estimate
\begin{align*}
K^*(T;\grm)&\ll (P^{7/8+\eps})^2(T^{2\del})^2\Biggl( \int_0^1|g(\alp)|^s\d\alp
\Biggr)^2\\
&\ll \left( P^{7/8+\eps}T^{2\del}\right)^2 \left( P^{s-4+\Del_s}\right)^2.
\end{align*}
On recalling (\ref{5.2}), we deduce that
\begin{align*}
\left( T^{16\Del_s-2}\right)^2P^2K^*(T;\grm)&\ll P^\eps \left( 
P^{\tfrac{15}{8}-2\del-15\Del_s}T^{16\Del_s-2+2\del}\right)^2(P^{u-4})^2\\
&=P^\eps \left( \frac{P^{15/16}}{T}\right)^{4-32\Del_s}
\left( \frac{T}{P}\right)^{4\del}(P^{u-4})^2 .
\end{align*}
Then since, by hypothesis, one has $P^{15/16}\le T\le P$ and $\Del_s<1/8$, we obtain
\begin{equation}\label{5.11}
\left( T^{16\Del_s-2}\right)^2P^2K^*(T;\grm)\ll (T/P)^{2\del}(P^{u-4})^2.
\end{equation}

\par Next, since $T<|g(\alp)|\le 2T$ when $\alp\in \grN(T)$, we find from (\ref{5.2}) and 
(\ref{5.10}) that when $T\ge P^{15/16}$, one has
\begin{align}
K^*(T;\grM)&\ll (T^{2-16\Del_s})^2\underset{\alp-\bet \in \grM}
{\int_{\grN(T)}\int_{\grN(T)}}|g(\alp )g(\bet)|^{s+2\del +16\Del_s-2}|G(\alp-\bet)|^2
\d\alp \d\bet \notag\\
&=(T^{2-16\Del_s})^2\Ome_0,\label{5.12}
\end{align}
where
\[
\Ome_0=\underset{\alp-\bet \in \grM}{\int_{\grN(T)}\int_{\grN(T)}}
|g(\alp )g(\bet)|^{u-2}|G(\alp-\bet)|^2\d\alp \d\bet .
\]
An application of H\"older's inequality shows that
\begin{equation}\label{5.13}
\Ome_0^{5+\del}\le \Ome_1\Ome_2\Ome_3^{3+\del}\sup_{(\alp,\bet)\in \grN(T)^2}
|g(\alp)g(\bet)|^{u-10-2\del},
\end{equation}
where we have written
\begin{equation}\label{5.14}
\Ome_1=\underset{\alp-\bet \in \grM}{\int_{\grN(T)}\int_{\grN(T)}}
|G(\alp-\bet)|^{5+\del}|g(\alp )|^u\d\alp \d\bet ,
\end{equation}
\begin{equation}\label{5.15}
\Ome_2=\underset{\alp-\bet \in \grM}{\int_{\grN(T)}\int_{\grN(T)}}
|G(\alp-\bet)|^{5+\del}|g(\bet )|^u\d\alp \d\bet ,
\end{equation}
and
\begin{equation}\label{5.16}
\Ome_3={\int_{\grN(T)}\int_{\grN(T)}}|g(\alp)g(\bet )|^u\d\alp \d\bet ,
\end{equation}

\par By a change of variable, we find from (\ref{5.14}) and (\ref{5.3}) that
\[
\Ome_1\le \biggl( \int_\grM |G(\tet)|^{5+\del}\d\tet \biggr)
 \biggl( \int_0^1|g(\alp)|^u\d\alp \biggr) \ll P^{1+\del}I.
\]
A symmetrical argument bounds the mean value $\Ome_2$ defined in (\ref{5.15}), and 
thus
\begin{equation}\label{5.17}
\Ome_1\Ome_2\ll \left(P^{1+\del}I\right)^2.
\end{equation}
On the other hand, it is immediate from (\ref{5.16}) that
\begin{equation}\label{5.18}
\Ome_3\le \biggl( \int_0^1|g(\alp)|^u\d\alp \biggr)^2=I^2.
\end{equation}
On substituting (\ref{5.17}) and (\ref{5.18}) within (\ref{5.13}), and noting (\ref{5.5}), 
we conclude thus far that
\begin{align*}
\Ome_0^{5+\del}&\ll \left(P^{1+\del}I\right)^2 \left( I^2\right)^{3+\del} 
T^{2(u-10-2\del)}\\
&\ll P^{2+2\del}I^{8+2\del}T^{2(u-10-2\del)}.
\end{align*}

\par We now find from (\ref{5.12}) that
\[
\left( T^{16\Del_s-2}\right)^2P^2K^*(T;\grM)\ll \left( I^{8+2\del}P^{12+4\del}
T^{2(u-10-2\del)}\right)^{1/(5+\del)}.
\]
Combining this estimate with (\ref{5.11}), and substituting into (\ref{5.9}) and thence into 
(\ref{5.8}), we discern that
\[
I_1(T)\ll (T/P)^\del P^{u-4}+\left( I^{4+\del}P^{6+2\del}
T^{u-10-2\del}\right)^{1/(5+\del)},
\]
so that, in view of our earlier observation that $u>10+2\del$, we obtain the relation
\[
\sum_{\substack{j=0\\ 2^j\le P^{1/16}}}^\infty I_1(2^{-j}P)\ll P^{u-4}+\left( 
I^{4+\del}P^{u-4}\right)^{1/(5+\del)}.
\]
Referring back to (\ref{5.6}) and (\ref{5.7}), we arrive at the upper bound
\[
I\ll P^{u-4}+\left( I^{4+\del}P^{u-4}\right)^{1/(5+\del)},
\]
whence $I\ll P^{u-4}$. This completes the proof of the lemma.
\end{proof}

\begin{corollary}\label{corollary5.2} Provided that $u\ge 11.95597$, one has
\begin{equation}\label{5.19}
\int_0^1|g(\alp;P,R)|^u\d\alp \ll P^{u-4}.
\end{equation}
In particular, the exponent $\Del_u=0$ is admissible.
\end{corollary}

\begin{proof} We apply Theorem \ref{theorem5.1} with $s=11.75$ and the admissible 
exponent $\Del_{11.75}=0.0128731$ supplied by Theorem \ref{theorem4.5}. We thus 
deduce that whenever
\[
u>11.75+16\Del_{11.75}=11.9559696,
\]
then the desired conclusion (\ref{5.19}) holds. This establishes that $\Del_u=0$ is 
admissible, completing the proof of the corollary.
\end{proof}

This corollary implies and is more or less equivalent to Theorem \ref{theorem1.2}. We 
performed extensive numerical computations in order to determine the optimal choice for 
$s$ in Theorem \ref{theorem5.1} in order that the value of $u$, with the exponent 
$\Del_u=0$, be minimised. It transpires that this optimal value is equal to $11.75$. We 
should remark that it is not altogether surprising that the optimal value occurs at a value 
of $s$ of the shape $s=12-2^{-j}$ for some non-negative integer $j$, because at each 
such value, it follows from Lemma \ref{lemma4.2} and the kind of arguments underlying 
Theorems \ref{theorem4.3} to \ref{theorem4.5} that there is a jump in the derivative of 
$\Del_s$ with respect to $s$. Here, we are thinking of $\Del_s$ as representing the least 
permissible admissible exponent as a function of $s$.

\bibliographystyle{amsbracket}
\providecommand{\bysame}{\leavevmode\hbox to3em{\hrulefill}\thinspace}

\end{document}